\documentclass[a4paper]{amsart}

\usepackage{amssymb}
\usepackage{latexsym}
\usepackage{amsmath}
\usepackage{euscript}

      \def\dC{{\mathbb C}}

      \def\dR{{\mathbb R}}

      \def\cC{{\mathcal C}}
      
   \def\cH{{\mathcal H}}   
      \def\cL{{\mathcal L}}

\def\cS{{\mathcal S}}

\def\cal H{{\mathcal H}}

\def\Re{{\text{\rm Re\,}}}
\def\Im{{\text{\rm Im\,}}}

\def\ran{{\text{\rm ran\,}}}
\def\dom{{\text{\rm dom\,}}}
\def\max{{\text{\rm max\,}}}
\def\min{{\text{\rm min\,}}}

\def\ess{{\text{\rm ess}}}

\def\phi{\varphi}

\newtheorem{theorem}{Theorem}[section]
\newtheorem*{thm*}{Theorem}

\newtheorem{corollary}[theorem]{Corollary}
\newtheorem{lemma}[theorem]{Lemma}

\theoremstyle{definition}
 
\newtheorem{remark}[theorem]{Remark}

\numberwithin{equation}{section}

\title[Elliptic differential operators with indefinite weights]{Spectral theory of elliptic differential operators with indefinite weights}

\author{Jussi Behrndt}

\address{Institut f\"ur Numerische Mathematik\\
Technische Universit\"at Graz \\
Steyrergasse 30\\
A-8010 Graz \\
Austria}
\email{behrndt@tugraz.at}

\begin{document}

\begin{abstract}
The spectral properties of a class of non-selfadjoint second order 
elliptic operators with indefinite weight functions 
on unbounded domains $\Omega$ are investigated.
It is shown that under an abstract regularity assumption the nonreal spectrum of
the associated elliptic operator in $L^2(\Omega)$ is bounded. In the special case
that $\Omega=\dR^n$ decomposes into subdomains $\Omega_+$ and $\Omega_-$
with smooth compact boundaries and the weight function is positive  
on $\Omega_+$ and negative on $\Omega_-$,
it turns out that the nonreal spectrum consists only of normal eigenvalues which can
be characterized with a Dirichlet-to-Neumann map. 
\end{abstract}

\maketitle

\section{Introduction}

The present paper is concerned with the spectral properties of 
partial differential operators associated to second order elliptic differential expressions of the form
\begin{equation}\label{tttt}
\cL f=\frac{1}{r}\,\ell (f),\qquad \ell(f)=-\sum_{j,k=1}^n \frac{\partial}{\partial x_j} a_{jk} \frac{\partial }{\partial x_k}f
+a f,
\end{equation}
with variable coefficients $a_{jk}$, $a$, and a weight function $r$ defined on some bounded or
unbounded domain $\Omega\subset\dR^n$, $n>1$.
It is assumed that the differential expression $\ell$ is formally symmetric and 
uniformly elliptic. The peculiarity in this paper is that 
the function $r$ is allowed to have different signs on subsets of positive Lebesgue
measure of $\Omega$.
For this reason $\cL$ is said to be an {\it indefinite} elliptic differential expression.

The differential expression $\ell$ in \eqref{tttt} gives rise to a selfadjoint unbounded
operator $A$ in the Hilbert space $L^2(\Omega)$ which is defined on the dense linear subspace
$\dom A=\{f\in H^1_0(\Omega):\ell(f)\in L^2(\Omega)\}$. The spectral properties of the elliptic differential operator $A$ depend on the geometry of $\Omega$ and the coefficients $a_{jk}$ and $a$, and are, at least 
from a qualitative point of view, well understood: The selfadjointness and ellipticity of $A$ imply 
that the spectrum of $\sigma(A)$ is contained in $\dR$ and that it is semibounded from below. 
If the domain $\Omega$ is bounded
or ``thin'' at $\infty$, then the resolvent of $A$ is compact and hence $\sigma(A)$ consists of a
sequence of eigenvalues with finite dimensional eigenspaces which accumulates
to $+\infty$; see, e.g., \cite{EE}. For general unbounded domains $\sigma(A)$ may also
contain continuous and essential spectrum of rather arbitrary form. However, if, e.g., the coefficients $a_{jk}$ and $a$ converge to a limit for $\vert x\vert\rightarrow\infty$,
then the essential spectrum of $A$ consists of a single unbounded interval.

In contrast to the selfadjoint case 
the spectral properties of the non-selfadjoint indefinite elliptic operator 
\begin{equation}\label{ttttt}
T=\frac{1}{r}A,\qquad \dom T=\dom A,
\end{equation} 
associated to the differential expression in \eqref{ttttt} are
much less understood, in particular, if the domain $\Omega$ is unbounded.
The case of a bounded domain $\Omega$ is discussed in, e.g., \cite{F90-1,F90}, 
where the point of view is similar to ours.
Further properties of indefinite elliptic operators on bounded domains, 
as, e.g., asymptotical behaviour of eigenvalues or Riesz basis properties of eigenfunctions have been studied 
(also for more general elliptic problems involving indefinite weights) in various papers. We mention here in particular the works
\cite{F88,F95,F00,F02,F09,FL96} of M. Faierman and \cite{P89,P95,P00,P09} of S.G. Pyatkov,
and, e.g., \cite{ADF97,DFM02,FM07}.

The main objective of the present paper is to study spectral properties of 
non-selfadjoint indefinite
elliptic operators of the form \eqref{ttttt} on unbounded domains.
Such problems are more difficult to investigate and a
purely abstract operator theoretic and functional analytic approach is  
insufficient in this situation (since, e.g., the essential
spectrum of $A$ is in general nonempty it is difficult to conclude that the spectrum of $T$ does
not cover the whole complex plane).
Therefore, in this paper we combine methods from the classical theory of elliptic differential equations 
with modern spectral and perturbation techniques for unbounded operators which are symmetric 
with respect to an indefinite inner product. Our investigations lead to new insights and 
results on the spectral properties of 
indefinite elliptic operators on unbounded domains, e.g., we
prove that under an abstract regularity assumption the nonreal spectrum of $T$ is bounded.
Furthermore, in the special case where
$\Omega=\dR^n$ decomposes into subdomains $\Omega_+$ and $\Omega_-$
with smooth compact boundaries such that the weight function $r$ is positive (negative) on $\Omega_+$ ($\Omega_-$, respectively) it is shown that the nonreal spectrum of $T$ consists only of normal eigenvalues which can
be characterized with Dirichlet-to-Neumann maps acting on interior and exterior domains.

The paper is organized as follows.
After the precise assumptions and basic facts explained in Section~\ref{222} the known case 
of a bounded domain $\Omega$ is discussed in Section~\ref{333} for completeness, see, e.g., \cite{CN93,F90,F95,P89}.
As one might expect it turns out that in this case the resolvent of $T$ is compact 
and hence $\sigma(T)$ consists only of eigenvalues with finite multiplicity.
Some additional facts on selfadjoint operators with finitely many negative squares 
in indefinite inner product spaces from \cite{CL89,L65,L82} imply that the nonreal spectrum of $T$ 
consists of at most finitely many eigenvalues. 
Section~\ref{444} deals with general unbounded domains.
If the spectrum or essential spectrum of $A$ is positive, then again 
abstract methods ensure that the nonreal spectrum of $T$ is bounded and consists of at most finitely many eigenvalues; cf. \cite{BP10,CL89,CN93,KMWZ03,L65,L82} and 
Theorems~\ref{known1} and \ref{known2}. 
One of our main results in the present paper states that without further assumptions on the operator $A$
the nonreal spectrum of $T$ remains bounded if a certain isomorphism $W$ which ensures the regularity of the critical point $\infty$ exists; cf. condition (I) in Theorem~\ref{bigthm}. 
In Section~\ref{omegar} the special case $\Omega=\dR^n$ with $r$ having negative sign outside a bounded set 
is studied. A sufficient condition in terms of the weight
function $r$ is given such that the nonreal spectrum of $T$ is bounded. A more detailed analysis is 
provided in Theorem~\ref{rnthm}, where a multidimensional variant of Glazmans decomposition method is used
to show that the nonreal spectrum of $T$ consist only of eigenvalues with finite multiplicity
which may accumulate to certain subsets of the real line. Finally, it is shown in Theorem~\ref{sigmat} 
how the nonreal spectrum of $T$
can be characterized with the help of Dirichlet-to-Neumann maps acting on interior and exterior domains
and a variant of Krein's resolvent formula for indefinite elliptic differential operators is obtained in Theorem~\ref{resformel}.

\section{Elliptic differential operators in $L^2(\Omega)$}\label{222}

In this preliminary section we define an elliptic differential expression $\cL$ with an indefinite
weight function on some domain $\Omega$ 
and we associate an unbounded differential operator in $L^2(\Omega)$ to $\cL$
which is selfadjoint with respect to an indefinite metric on $L^2(\Omega)$; cf. 
Theorem~\ref{tthm}.

\subsection{The elliptic differential expression}
Let $\Omega\subset\dR^n$ be a domain and let $\ell$ be the "formally selfadjoint" uniformly 
elliptic second order differential expression
\begin{equation}\label{ellell}
(\ell f)(x):=-\sum_{j,k=1}^n \left( \frac{\partial}{\partial x_j} a_{jk} \frac{\partial f}{\partial x_k}\right)(x)+(a\!\,f)(x),\quad x\in\Omega,
\end{equation}
with bounded coefficients $a_{jk}\in C^\infty(\Omega)$ satisfying $a_{jk}(x)=\overline{a_{kj}(x)}$ 
for all $x\in\Omega$ and $j,k=1,\dots,n$, the function $a\in L^\infty(\Omega)$ is real valued and 
\begin{equation*}
\sum_{j,k=1}^n a_{jk}(x)\xi_j\xi_k\geq C\sum_{k=1}^n\xi_k^2
\end{equation*}
holds for some $C>0$, all $\xi=(\xi_1,\dots,\xi_n)^\top\in\dR^n$ and $x\in\Omega$. 

In the following we investigate operators induced by the second order elliptic differential 
expression $\cL$ with the indefinite weight $r$ defined by
\begin{equation*}
(\cL f)(x):=\frac{1}{r(x)}\,(\ell f)(x),\qquad x\in\Omega.
\end{equation*}
Throughout this paper it is assumed that $r$ is a real valued function such that $r,r^{-1}\in L^\infty(\Omega)$ 
and each of the sets
\begin{equation}\label{omegapm}
 \Omega_+:=\bigl\{x\in\Omega: r(x)>0\bigr\}\quad\text{and}\quad \Omega_-:=\bigl\{x\in\Omega: r(x)<0\bigr\}
\end{equation}
has positive Lebesgue measure. Observe that $\Omega\backslash(\Omega_+\cup\Omega_-)$ is a Lebesgue null set. 
The restriction of the weight function $r$ onto $\Omega_\pm$ is denoted by $r_\pm$. 
Similarly, for a function $f$ defined on $\Omega$ the restriction onto $\Omega_\pm$ is denoted by
$f_\pm$. Moreover, $\ell_\pm$ and $\cL_\pm$ stand for the restrictions of the differential
expressions $\ell$ and $\cL$ onto $\Omega_\pm$.

\subsection{Differential operators in $L^2(\Omega)$ associated to $\ell$ and $\cL$}

To the differential expression $\ell$ we associate the elliptic differential operator
\begin{equation}\label{a}
Af:=\ell (f),\qquad \dom A=\bigl\{f\in H^1_0(\Omega):\ell (f) \in L^2(\Omega)\bigr\},
\end{equation}
where $H^1_0(\Omega)$ stands for the closure of $C_0^\infty(\Omega)$ in the Sobolev space $H^1(\Omega)$.
It is well known that $A$ is an unbounded selfadjoint operator in the Hilbert space $(L^2(\Omega),(\cdot,\cdot))$ with spectrum
semibounded from below by $\text{essinf}\,a$. This can be seen, e.g., with the help of
the sesquilinear form associated to $\ell$ and the first representation theorem from \cite{Kato}.

Besides the Hilbert space inner product $(\cdot,\cdot)$ in $L^2(\Omega)$ we will make use of the indefinite
inner product
\begin{equation}\label{indefprod}
[f,g]:=\int_\Omega f(x)\overline{g(x)}\,r(x)\,dx,\qquad f,g\in L^2(\Omega).
\end{equation}
The space $(L^2(\Omega),[\cdot,\cdot])$ is a so-called Krein space; cf. \cite{AI89,B74,K70,L65,L82}. 
Observe that $[\cdot,\cdot]$ is nonpositive on functions with support in $\Omega_-$
and nonnegative on functions with support in $\Omega_+$. 
Note also that the assumptions $r\in L^\infty(\Omega)$ and $r^{-1}\in L^\infty(\Omega)$ imply that
the multiplication operator $Rf=rf$, $f\in L^2(\Omega)$, is an isomorphism in $L^2(\Omega)$
with inverse $R^{-1}f=r^{-1}f$, $f\in L^2(\Omega)$. In particular,
$\ell(f)\in L^2(\Omega)$ if and only if $\cL(f)\in L^2(\Omega)$. 
Furthermore, the inner products $(\cdot,\cdot)$ and
$[\cdot,\cdot]$ are connected via 
\begin{equation}\label{rfg}
[f,g]=(Rf,g)\quad\text{and}\quad (f,g)=[R^{-1}f,g]\quad\text{for}\quad f,g\in L^2(\Omega).
\end{equation}

Next we introduce the differential operator $T$ associated to the indefinite
elliptic expression $\cL$ and we summarize some of its properties. The following theorem is 
a direct consequence of \eqref{rfg} and the selfadjointness of $A$. 

\begin{theorem}\label{tthm}
The differential operator
\begin{equation}\label{t}
Tf:=\cL (f),\qquad \dom T=\bigl\{f\in H^1_0(\Omega):\cL (f) \in L^2(\Omega)\bigr\},
\end{equation}
is selfadjoint with respect to the Krein space inner product $[\cdot,\cdot]$ in $L^2(\Omega)$,
and $T$ is connected with the elliptic differential operator $A$ in \eqref{a} via
\begin{equation*}
T= R^{-1}A\qquad\text{and}\qquad A=R T.
\end{equation*}
\end{theorem}

We remark that the adjoint of an (unbounded) operator with respect to a Krein space inner product
is defined in the same way as with respect to a usual scalar product. Here the
adjoint $T^+$ of $T$ with respect to $[\cdot,\cdot]$ can equivalently be defined by $T^+:=R^{-1}A^*=R^{-1}A$, where $^*$ denotes the adjoint with respect to $(\cdot,\cdot)$.
In particular, this implies $[Tf,g]=[f,Tg]$ for all $f,g\in\dom T$.

We also point out that the spectrum of an operator which is selfadjoint in the Krein space $(L^2(\Omega),[\cdot,\cdot])$ 
can be quite arbitrary. In particular, the spectrum is in general not a subset of $\dR$ 
and simple examples show that the spectrum can be empty or cover the whole complex plane.
However, the nonreal spectrum is necessarily symmetric with respect to the real line.

\subsection{Spectral points of closed operators}

Let $S$ be a closed operator in a Hilbert space. The \emph{resolvent set} $\rho(S)$ 
of $S$ consists of all $\lambda\in\dC$ such that $S-\lambda$ is bijective.
The complement of $\rho(S)$ in $\dC$ is the \emph{spectrum} $\sigma(S)$ of $S$. The \emph{point spectrum}
$\sigma_p(S)$ is the set of eigenvalues of $S$, i.e., those $\lambda\in\dC$ for which $S-\lambda$ 
is not injective. An eigenvalue $\lambda$ is said to be \emph{normal} if $\lambda$ is an isolated
point of $\sigma(S)$ and its (algebraic) multiplicity is finite. The {\it essential spectrum} $\sigma_{\text{\rm ess}}(S)$ 
consists of those points $\lambda\in\dC$ for which $S-\lambda$ is not a Semi-Fredholm operator. Recall that
the essential spectrum is stable under compact and relative compact perturbations; cf. \cite{EE,Kato}.
If $S$ is a selfadjoint operator, then $\sigma_{\text{\rm ess}}(S)$ consists of the accumulation 
points of $\sigma(S)$ and the isolated eigenvalues of infinite multiplicity; the
set of normal eigenvalues is the complement of $\sigma_{\text{\rm ess}}(S)$ in $\sigma(S)$.
Recall that the eigenvalues of a selfadjoint operator are semisimple.
We say that the \emph{positive} (\emph{negative}) 
\emph{spectrum} of (a not necessarily selfadjoint operator) $S$ has \emph{infinite multiplicity} if  
$\sigma(S)\cap(0,+\infty)$ ($\sigma(S)\cap (-\infty,0)$, respectively) contains infinitely many eigenvalues or points of the essential
spectrum of $S$.

\section{Spectral properties of indefinite elliptic operators on bounded domains}\label{333}

In this section we study the spectral properties of the indefinite elliptic operator $T$ in Theorem~\ref{tthm} in the case that $\Omega$ is a bounded
domain in $\dR^n$. Throughout this section it will be tacitly assumed that $\Omega$ 
is bounded, but no further (regularity) assumptions on the boundary are imposed. 

Let us first recall the following well-known theorem on the qualitative spectral properties
of the selfadjoint elliptic operator $A$ which is essentially a consequence of the compactness of the embedding
of $H^1_0(\Omega)$ into $L^2(\Omega)$ (see, e.g., \cite[Theorem~7.1]{W87}), the
ellipticity of $\ell$ and the boundedness of the
coefficient $a$.

\begin{theorem}\label{speca}
The spectrum of $A$ is bounded from below and consists of normal semisimple eigenvalues which
accumulate to $+\infty$.
\end{theorem}

The principal result in this section is the following theorem, which is well known and follows from the more general 
and abstract considerations 
in \cite{F90,F95,P89} and \cite{CL89,L65,L82}. For the convenience of the reader a short proof is included.

\begin{theorem}\label{boundedthm}
The spectrum of $T$ consists of normal eigenvalues which
accumulate to $+\infty$ and $-\infty$. The nonreal spectrum of $T$ is bounded and consists
of at most finitely many normal eigenvalues which are symmetric with respect to the real line.
\end{theorem}

Before we prove this theorem a preparatory lemma on the resolvent set of $T$ will be proved.

\begin{lemma}\label{rhot}
The set $\rho(T)$ is nonempty.
\end{lemma}

\begin{proof}
If $0$ is not a normal eigenvalue of $A$, then $0\in\rho(A)$ and 
it follows from $T=R^{-1}A$ that $T^{-1}=A^{-1}R\in\cL(L^2(\Omega))$
holds, i.e., $0\in\rho(T)$. Therefore, assume that $0\in\sigma(A)$, that is, $0$ is an isolated
eigenvalue of finite multiplicity of $A$ by Theorem~\ref{speca}. The restriction 
\begin{equation*}
B:=A\upharpoonright (\dom A\cap(\ker A)^\bot)
\end{equation*}
of $A$ on the orthogonal complement of $\ker A$ in $L^2(\Omega)$
is regarded as a nondensely defined symmetric operator in $L^2(\Omega)$ with finite equal defect numbers. 
Note that $B$ is injective and that $\ran B=(\ker A)^\bot$ is closed and has finite codimension. 
Hence there exists a selfadjoint operator $\widetilde A$ in $L^2(\Omega)$ which is an extension of $B$ such that $0\in\rho(\widetilde A)$.
Furthermore, since $B$ is a finite dimensional restriction of both $A$ and $\widetilde A$ it follows that
\begin{equation*}
\dim\ran\bigl((A-\lambda)^{-1}-(\widetilde A-\lambda)^{-1}\bigr)\leq\dim\ker A<\infty 
\end{equation*}
holds for all $\lambda\in\rho(A)\cap\rho(\widetilde A)$ and hence $\sigma(\widetilde A)$ is semibounded from
below and consists of normal eigenvalues. 

The operator $\widetilde T:=R^{-1}\widetilde A$ is a selfadjoint operator in the Krein space $(L^2(\Omega),[\cdot,\cdot])$
and from $0\in\rho(\widetilde A)$ we conclude $0\in\rho(\widetilde T)$. Furthermore, since $\widetilde A$
is semibounded from below and $[\widetilde Tf,g]=(\widetilde A f,g)$ holds for all 
$f,g\in\dom \widetilde T=\dom\widetilde A$, it follows that the form 
$[\widetilde T\cdot,\cdot]$ has finitely many negative squares. It is easy to see that $\widetilde T$ and
$T$ are both finite dimensional extensions of the nondensely defined operator $S:=R^{-1} B$.
Now $\rho(T)\not=\emptyset$ follows form a slight modification of \cite[Proposition~1.1]{CL89}, 
see also \cite[Corollary~2.5]{ABT08}.
\end{proof}

\begin{proof}[Proof of Theorem~\ref{boundedthm}]
Observe that by Theorem~\ref{speca} the resolvent $(A-\lambda)^{-1}$ is compact for all 
$\lambda\in\rho(A)$ and that Lemma~\ref{rhot} implies $\rho(T)\cap\rho(A)\not=\emptyset$. A simple
computation shows that the relation 
\begin{equation*}
 (T-\lambda)^{-1}=(A-\lambda)^{-1}R-\lambda(A-\lambda)^{-1}(I-R)(T-\lambda)^{-1}
\end{equation*}
holds for all $\lambda\in\rho(A)\cap\rho(T)$,
and since the right hand side is a compact operator the same holds for
the left hand side. Hence $\sigma(T)$ consists of normal eigenvalues. 
As the negative spectrum of $A$ consists of at most finitely many normal eigenvalues
the form $[T\cdot,\cdot]=(A\cdot,\cdot)$ has finitely many negative squares and it follows
from the general results in \cite{CL89,L65,L82} that the nonreal spectrum of $T$ consists 
of at most finitely many normal eigenvalues which are symmetric with respect to the real line.
Finally, the assumption that the sets $\Omega_+$ and $\Omega_-$ in \eqref{omegapm} have positive Lebesgue measure
imply that the indefinite inner product $[\cdot,\cdot]$ in \eqref{indefprod} has infinitely many positive and
negative squares. The reasoning in \cite[Proof of Proposition~1.8]{CL89} shows
that the positive spectrum of $T$, as well as the negative spectrum of $T$ is of infinite multiplicity, and hence the 
real eigenvalues of $T$ accumulate to $+\infty$ and $-\infty$.
\end{proof}

\section{Spectral properties of indefinite elliptic operators on unbounded domains}\label{444}

In this section we study the spectral properties of the indefinite elliptic operator $T$ in \eqref{t} on an unbounded
domain $\Omega\subset\dR^n$. 
Since for an unbounded domain the embedding of $H^1_0(\Omega)$ into $L^2(\Omega)$ is in general not compact
also the resolvent of 
the selfadjoint operator $A$ in \eqref{a} is in general not compact and hence essential spectrum may occur.
Only the following weaker variant of Theorem~\ref{speca} holds.

\begin{theorem}\label{speca2}
The spectrum of $A$ is bounded from below and accumulates to $+\infty$.
\end{theorem}
 
If the lower bound $\min\sigma(A)$ of the spectrum of $A$ or the lower bound $\min\sigma_\ess(A)$ of the essential spectrum 
of $A$ is positive, then it is known that $T=R^{-1}A$ is positive in the Krein space $(L^2(\Omega),[\cdot,\cdot])$
or has a finite number of negative squares, respectively. 
For the convenience of the reader we recall these and some other facts in Theorem~\ref{known1} and Theorem~\ref{known2} below. The proofs of the statements are essentially contained in \cite{B74,CL89,L65,L82}, see also \cite[Theorem~3.3]{KMWZ03}, \cite[Theorem~3.1]{BT07}
and \cite[Proposition~1.6]{CN95-2}. 

\begin{theorem}\label{known1}
If $\min\sigma(A)>0$, then the spectrum of $T$ is real, $0\in\rho(T)$, and $T$
has positive and negative spectrum,
both of infinite multiplicities.
\end{theorem}

\begin{theorem}\label{known2}
If $\min\sigma_\ess(A)>0$, then the essential spectrum of $T$ is real and $T$
has positive and negative spectrum,
both of infinite multiplicities.  The nonreal spectrum of $T$
is bounded and consists of at most finitely many normal eigenvalues which are symmetric with respect to the real line.
\end{theorem}

In the next step the assumption $\min\sigma_\ess(A)>0$ will be dropped. 
The following considerations and Theorem~\ref{bigthm} below are partly inspired by general results on selfadjoint operators in Krein spaces and the regularity of the critical point $\infty$ from \cite{C85,CN93,L82}
and \cite[Proof of Theorem 5.4]{BP10}.
Fix some $\nu<\min\sigma(A)$ and define the space $\cH_s$, $s\in[0,2]$, as the domains of the $\frac{s}{2}$-th powers of the positive operator $A-\nu$,
\begin{equation*}
\cH_s:=\dom \bigl((A-\nu)^\frac{s}{2}\bigr),\qquad s\in [0,2].
\end{equation*}
Note that $\cH=\cH_0$, $\dom A=\cH_2$, and the form domain of $A$ is $\cH_1$. The spaces $\cH_s$
become Hilbert spaces when they are equipped with the usual inner products, the induced 
topologies do not depend on the particular choice of $\nu<\min\sigma(A)$; cf. \cite{Kato}.

The following theorem is one of the main results of the present paper. Under an additional
abstract condition from \cite{CN93} it will be shown that the nonreal spectrum of the indefinite elliptic operator
is bounded. Roughly speaking, this condition is satisfied in special situations 
when chosing $W=R$; cf. Lemma~\ref{sufflem} in the next section.

\begin{theorem}\label{bigthm}
Assume that $\min\sigma_\ess(A)\leq 0$ and that the following condition holds:
\begin{itemize}
 \item [(I)] There exists an isomorphism $W$ in $L^2(\Omega)$ such that $RW$ is positive in $L^2(\Omega)$
and $W\cH_s\subset\cH_s$ holds for some $s\in (0,2]$. 
\end{itemize}
Then the nonreal spectrum of $T$ is bounded.
\end{theorem}

\begin{proof}
{\bf 1.} In this step of the proof we construct an indefinite elliptic operator $T_\eta$ which is a bounded 
perturbation of the indefinite elliptic operator $T$ and which induces (via its spectral decomposition) 
a new equivalent norm $\Vert\cdot\Vert_\sim$ on $L^2(\Omega)$.
 
For this fix some $\eta<\min\sigma(A)$ and consider the elliptic differential operator $A_\eta$ defined by
\begin{equation*}
A_\eta f:=(A-\eta)f=-\sum_{j,k=1}^n \frac{\partial}{\partial x_j} a_{jk} \frac{\partial f}{\partial x_k} +( a-\eta)f,
\quad f\in\dom A_\eta=\dom A.
\end{equation*}
Clearly $A_\eta$ is a positive selfadjoint operator in the Hilbert 
space $L^2(\Omega)$ and hence 
the indefinite elliptic operator 
\begin{equation*}
T_\eta f:=\frac{1}{r}
\left(-\sum_{j,k=1}^n \frac{\partial}{\partial x_j} a_{jk} \frac{\partial f}{\partial x_k} +(a-\eta)f\right),\quad f\in\dom T_\eta=\dom A_\eta,
\end{equation*}
is nonnegative in the Krein space $(L^2(\Omega),[\cdot,\cdot])$, the spectrum 
$\sigma(T_\eta)$ is a subset of $\dR$ and $0\in\rho(T_\eta)$; cf. Theorem~\ref{known1}.
Note that $T_\eta$ and $T$ are connected via
\begin{equation}\label{tetaaeta}
T_\eta=R^{-1} A_\eta= R^{-1} A- \eta R^{-1}=T-V,\qquad V:=\eta R^{-1},
\end{equation}
and that the perturbation term $V$ in \eqref{tetaaeta} is bounded.

By \cite{L65,L82} $T_\eta$ possesses a spectral function defined for all bounded subintervals of 
the real line. 
As a consequence of condition (I) and \cite[Theorem 2.1 (iii)]{CN93} (see also \cite{C85}) 
it follows that $\infty$ is
not a singular critical point of the operator $T_\eta$ and therefore also the spectral projections $E_+$ and $E_-$
corresponding to the intervals $(0,+\infty)$ and $(-\infty,0)$ exist. Moreover, as $T_\eta$ is nonnegative in $(L^2(\Omega),[\cdot,\cdot])$ 
the spectral subspaces $(E_\pm L^2(\Omega),\pm [\cdot,\cdot])$ are both Hilbert spaces
and $L^2(\Omega)$ can be decomposed in
\begin{equation}\label{decol2}
L^2(\Omega)=E_+ L^2(\Omega) [\dot +]E_- L^2(\Omega).
\end{equation}
We point out that the subspaces $E_\pm L^2(\Omega)$ differ from $L^2(\Omega_\pm)$ and that 
within this proof the subscripts $\pm$ are used in the sense of \eqref{decol2}.
From the properties of the spectral function it follows that
$T_\eta$ has diagonal form with respect to the space decomposition \eqref{decol2},
\begin{equation*}
T_\eta=\begin{pmatrix} T_{\eta,+} & 0 \\ 0 & T_{\eta_-}\end{pmatrix}, 
\end{equation*}
and  that the spectrum of $T_{\eta,\pm}$ is contained in $\dR^\pm$. The perturbation term
$V=\eta R^{-1}$ in \eqref{tetaaeta} admits the matrix representation 
\begin{equation*}
V=\begin{pmatrix} V_{11} & V_{12} \\ V_{21} & V_{22}\end{pmatrix}
\end{equation*}
with respect to the decomposition \eqref{decol2}.
Together with \eqref{tetaaeta}  we then have
\begin{equation}\label{tetav}
T=T_\eta+V=T_\eta+\begin{pmatrix} V_{11} & V_{12} \\ V_{21} & V_{22}\end{pmatrix}.
\end{equation}

In the following we write functions $x,y\in L^2(\Omega)$ in the form $x=x_++x_-$ and $y=y_++y_-$, where $x_\pm,y_\pm\in 
E_\pm L^2(\Omega)$; cf. \eqref{decol2}. 
We emphasize that $x_\pm$ are the components of $x$ with respect to the space decomposition \eqref{decol2}
and that $x_\pm$ do not coincide with  the restrictions of the function $x$ onto $\Omega_\pm$. 
Since the spectral subspaces $(E_\pm L^2(\Omega),\pm [\cdot,\cdot])$ are Hilbert spaces 
the inner product $(\cdot,\cdot)_\sim$ defined by
\begin{equation*}
(x,y)_\sim:=[x_+,y_+]-[x_-,y_-],\qquad x,y\in L^2(\Omega),
\end{equation*}
is positive definite. Furthermore, this scalar product is connected with the usual scalar product $(\cdot,\cdot)$ 
on $L^2(\Omega)$ via
\begin{equation*}
(x,y)_\sim=[E_+ x,E_+y]-[E_- x,E_- y]=[(E_+-E_-)x,y]=(R(E_+-E_-)x,y).
\end{equation*}
Therefore, as $R(E_+-E_-)$ is an isomorphism, the norms $\Vert\cdot\Vert$ and $\Vert\cdot\Vert_\sim$ 
induced by the scalar products $(\cdot,\cdot)$ and $(\cdot,\cdot)_\sim$, respectively, are equivalent. 
In particular, with $\nu:=\Vert R(E_+-E_-)\Vert^{-1}$
we have 
\begin{equation}\label{cvbn}
\Vert x\Vert\leq\sqrt{\nu}\Vert x\Vert_\sim\qquad\text{for all}\,\,\, x\in L^2(\Omega).
\end{equation}

\vskip 0.2cm\noindent
{\bf 2.} In this step it will be shown that for sufficiently large $\vert\mu\vert$, $\mu\in\dC\backslash\dR$ with $\Re\mu\leq0$ the operator
$T_{\eta,+}+ V_{11} -\mu$ is invertible and that the estimates 
\begin{equation}\label{normest}
\Vert (T_{\eta,+}+ V_{11} -\mu)^{-1}\Vert_\sim<\frac{1}{2}\quad\text{and}\quad
\Vert (T_{\eta,+}+ V_{11} -\mu)^{-1}V_{12}\Vert_\sim<\frac{1}{2}
\end{equation}
hold. By replacing $V_{11}$ and $V_{12}$ in the reasoning below with $V_{22}$ and $V_{21}$, respectively, it follows
that for sufficiently large $\vert\mu\vert$, $\mu\in\dC\backslash\dR$ with $\Re\mu\geq0$ the operator
$T_{\eta,-}+ V_{22} -\mu$ is invertible and that the estimates 
\begin{equation*}
\Vert (T_{\eta,-}+ V_{22} -\mu)^{-1}\Vert_\sim<\frac{1}{2}\quad\text{and}\quad
\Vert (T_{\eta,-}+ V_{22} -\mu)^{-1}V_{21}\Vert_\sim<\frac{1}{2}
\end{equation*}
are valid.

In the following we assume that the entry $V_{12}$ in the perturbation term $V$ is nonzero (otherwise 
the first estimate in \eqref{normest} follows with $\delta>2$ and $\tau=\delta+\Vert V_{11}\Vert_\sim$ in the
argument below; the second estimate is trivial). Choose $\delta>0$ such that
\begin{equation}\label{deltacond}
\delta+\Vert V_{11}\Vert_\sim>\max\left\{\frac{2+\Vert V_{11}\Vert_\sim}{\Vert  V_{12}\Vert_\sim} ,
2+\frac{\Vert  V_{11}\Vert_\sim}{\Vert V_{12}\Vert_\sim}\right\}
\end{equation}
and define the constant $\tau$ by
\begin{equation}\label{tau}
\tau:=\bigl(\delta+\Vert V_{11}\Vert_\sim\bigr)\Vert V_{12}\Vert_\sim
\end{equation}
Let $\mu\in\dC\backslash\dR$ with $\Re\mu\leq 0$ and $\vert\mu\vert>\tau$.
Since $\sigma(T_{\eta,+})\subset\dR^+$ it is clear that 
\begin{equation*}
 \text{dist}\,(\mu,\sigma(T_{\eta_+}))>\tau
\end{equation*}
holds and therefore we have $\Vert (T_{\eta,+}-\mu)^{-1}\Vert_\sim <\tau^{-1}$.
This implies 
\begin{equation*}
\Vert V_{11} (T_{\eta,+}-\mu)^{-1}\Vert_\sim <\frac{1}{\tau} \Vert V_{11} \Vert_\sim
\end{equation*}
and it follows from \eqref{deltacond} and \eqref{tau} that $\Vert V_{11} (T_{\eta,+}-\mu)^{-1}\Vert_\sim <1$.
Therefore the operator $I+ V_{11} (T_{\eta,+}-\mu)^{-1}$ is boundedly invertible and the norm of the inverse 
can be estimated by
\begin{equation*}
\Vert (I+ V_{11} (T_{\eta,+}-\mu)^{-1})^{-1}\Vert_\sim <
\left(1-\frac{1}{\tau}\Vert V_{11} \Vert_\sim\right)^{-1}
\end{equation*}
It follows that also the operator
\begin{equation*}
T_{\eta,+}+ V_{11} -\mu=\bigl(I+V_{11}(T_{\eta,+}-\mu)^{-1}\bigr)(T_{\eta,+}-\mu)
\end{equation*}
is boundedly invertible and we conclude
\begin{equation}\label{norm0}
\Vert (T_{\eta,+}+ V_{11} -\mu)^{-1}\Vert_\sim<\frac{1}{\tau}\left(1-\frac{1}{\tau}\Vert V_{11} \Vert_\sim\right)^{-1}
=\frac{1}{\tau-\Vert V_{11} \Vert_\sim}.
\end{equation}
Since by \eqref{deltacond} and \eqref{tau} $\tau-\Vert V_{11} \Vert_\sim>2$ we obtain the first estimate in \eqref{normest}.
Furthermore, as a consequence of \eqref{deltacond} and \eqref{tau} we have
\begin{equation*}
\frac{\Vert V_{12} \Vert_\sim}{\tau-\Vert V_{11} \Vert_\sim}=
\frac{\Vert V_{12} \Vert_\sim}{(\delta+\Vert V_{11}\Vert_\sim)
\Vert V_{12} \Vert_\sim-\Vert V_{11}\Vert_\sim}<\frac{1}{2}
\end{equation*}
and therefore \eqref{norm0} yields the second estimate in \eqref{normest},
\begin{equation*}
\Vert (T_{\eta,+}+ V_{11} -\mu)^{-1} V_{12} \Vert_\sim<
\frac{\Vert V_{12} \Vert_\sim}{\tau-\Vert V_{11} \Vert_\sim}<\frac{1}{2}.
\end{equation*}

\vskip 0.2cm\noindent
{\bf 3.} Next we verify the inequality 
\begin{equation}\label{goodinq}
\Vert (T-\mu)x\Vert_\sim\geq \left(
\sqrt{\left(1+\frac{\nu}{\vert\Im\mu\vert}\right)^2+1}-\left(1+\frac{\nu}{\vert\Im\mu\vert}\right)\,\right) \Vert x\Vert_\sim
\end{equation}
for $x\in\dom T$ and all sufficiently large $\vert\mu\vert$, $\mu\in\dC\backslash\dR$.
Observe first that for $x\in \dom T$  we have
\begin{equation*}
[(T-\mu)x,x]=[(T-\Re\mu)x,x]-i\Im\mu[x,x]
\end{equation*}
which together with \eqref{cvbn} implies
\begin{equation*}
\vert \Im\mu\vert \,\vert[ x,x]\vert\leq\vert[(T-\mu)x,x]\vert
\leq \Vert (T-\mu)x\Vert\Vert x\Vert\leq \nu\Vert (T-\mu)x\Vert_\sim\Vert x\Vert_\sim
\end{equation*}
and hence
\begin{equation}\label{est1}
 \frac{\nu}{\vert\Im\mu\vert} \Vert (T-\mu)x\Vert_\sim\Vert x\Vert_\sim \geq \pm [x,x].
\end{equation}

On the other hand, when we consider the equation $(T-\mu)x=y$ with $x=x_++x_-$, $y=y_++y_-$, 
$x_\pm,y_\pm\in E_\pm L^2(\Omega)$, that is,
\begin{equation*}
\begin{split}
(T_{\eta,+}+V_{11}-\mu)x_++V_{12}x_-&=y_+\\
V_{21}x_++(T_{\eta,-}+V_{22}-\mu)x_-&=y_-
\end{split}
\end{equation*}
(see \eqref{tetav}),
then we conclude with the help of the estimates from step 2 that for sufficiently large $\vert\mu\vert$, 
$\mu\in\dC\backslash\dR$ with $\Re\mu\leq 0$
\begin{equation*}
\begin{split}
\Vert x_+\Vert_\sim&\leq \Vert (T_{\eta_+}+V_{11}-\mu)^{-1}y_+ \Vert_\sim+ 
\Vert (T_{\eta_+}+ V_{11}-\mu)^{-1} V_{12}x_- \Vert_\sim\\
&\leq \frac{1}{2}\Vert y_+\Vert_\sim +\frac{1}{2} \Vert x_-\Vert_\sim
\leq \frac{1}{2}\Vert y_+\Vert_\sim +\frac{1}{2}\Vert x\Vert_\sim
\end{split}
\end{equation*}
holds and that for sufficiently large $\vert\mu\vert$, $\mu\in\dC\backslash\dR$
with $\Re\mu\geq 0$
\begin{equation*}
\begin{split}
\Vert x_-\Vert_\sim&\leq \Vert (T_{\eta_-}+V_{22}-\mu)^{-1}y_- \Vert_\sim+ 
\Vert (T_{\eta_-}+ V_{22}-\mu)^{-1} V_{21}x_+ \Vert_\sim\\
&\leq \frac{1}{2}\Vert y_-\Vert_\sim +\frac{1}{2} \Vert x_+\Vert_\sim
\leq \frac{1}{2}\Vert y_-\Vert_\sim +\frac{1}{2}\Vert x\Vert_\sim
\end{split}
\end{equation*}
holds. Since $\Vert y_\pm\Vert_\sim\leq\Vert y\Vert_\sim=\Vert(T-\mu)x\Vert_\sim$ we have 
\begin{equation}\label{esty}
\Vert x_\pm\Vert_\sim^2\leq \frac{1}{4}\Vert (T-\mu)x\Vert_\sim^2 +\frac{1}{4}\Vert x\Vert_\sim^2
+\frac{1}{2}\Vert (T-\mu)x\Vert_\sim \Vert x\Vert_\sim
\end{equation}
for sufficiently large $\vert\mu\vert$, $\mu\in\dC\backslash\dR$ with $\Re\mu\leq 0$ and $\Re\mu\geq 0$, respectively.
From $\Vert x_+\Vert_\sim^2+\Vert x_-\Vert_\sim^2=\Vert x\Vert_\sim^2$ we obtain
\begin{equation*}
\pm[x,x]=\pm\Vert x_+\Vert_\sim^2\mp\Vert x_-\Vert_\sim^2=\Vert x\Vert_\sim^2- 2 \Vert x_\mp\Vert_\sim^2
\end{equation*}
and together with \eqref{esty} we conclude
\begin{equation*}
\pm [x,x]\geq \frac{1}{2}\Vert x\Vert_\sim^2-\frac{1}{2}\Vert (T-\mu)x\Vert_\sim^2 -\Vert (T-\mu)x \Vert_\sim
\Vert x\Vert_\sim
\end{equation*}
for sufficiently large $\vert\mu\vert$, $\mu\in\dC\backslash\dR$ with $\Re\mu\geq 0$ and $\Re\mu\leq 0$, respectively.
Together with \eqref{est1} this leads to
\begin{equation*}
\frac{\nu}{\vert \Im\mu\vert}\Vert (T-\mu)x\Vert_\sim\Vert x\Vert_\sim\geq\frac{1}{2}\Vert x\Vert_\sim^2-\frac{1}{2}\Vert (T-\mu)x\Vert_\sim^2 -\Vert (T-\mu)x \Vert_\sim\Vert x\Vert_\sim,
\end{equation*}
for all sufficiently large $\vert\mu\vert$, $\mu\in\dC\backslash\dR$.
In other words, $\Vert (T-\mu)x\Vert_\sim$ satisfies the quadratic inequality
\begin{equation*}
\Vert (T-\mu)x\Vert_\sim^2+2\left(1+\frac{\nu}{\vert\Im\mu\vert}\right)\Vert x\Vert_\sim\Vert (T-\mu)x\Vert_\sim-\Vert x\Vert_\sim^2
\geq 0.
\end{equation*}
Hence it follows that \eqref{goodinq}
holds for all $x\in\dom T$ and all $\mu\in\dC\backslash\dR$ with $\vert\mu\vert$ sufficiently large.

\vskip 0.2cm\noindent
{\bf 4.} Let $\lambda\in\dC\backslash\dR$ such that \eqref{goodinq} is satisfied with $\mu=\lambda$ and $\mu=\bar\lambda$.
Then we have $\ker(T-\lambda)=\{0\}$ and $\ran(T-\lambda)$ is closed as $T$ is closed. 
Furthermore, since $T$ is selfadjoint 
in the Krein space $(L^2(\Omega),[\cdot,\cdot])$ it is clear that
$\ran(T-\lambda)^{[\bot]}=\ker(T-\bar\lambda)$
holds. As $\Vert(T-\bar\lambda)x\Vert_\sim$, $x\in\dom T$, satisfies the same estimate as $\Vert(T-\lambda)x\Vert_\sim$ in \eqref{goodinq} this 
implies that 
also $\ker(T-\bar\lambda)$ is trivial. Therefore $T-\lambda$ is bijective, i.e., $\lambda\in\rho(T)$. Since this is true for every $\lambda=\mu$ which satisfies 
\eqref{goodinq} we conclude that the nonreal spectrum of $T$ is bounded.
\end{proof}

\section{Spectral properties of indefinite elliptic operators on $\dR^n$}\label{omegar}

In this section we consider the case $\Omega=\dR^n$ and we assume that the subsets $\Omega_\pm=\{x\in\dR^n:\pm r(x)>0\}$
consist of finitely many connected components with compact smooth boundaries. In particular, this 
implies that one of the sets $\Omega_\pm$ is bounded and one is unbounded, and that the boundaries $\partial\Omega_+$ and
$\partial\Omega_-$ coincide. Here and in the following we discuss the case
where $\Omega_-$ is unbounded and $\Omega_+$ is bounded and we denote the boundary $\partial\Omega_\pm$ by $\cC$. 
The simple modifications of the results below to 
the other case are left to the reader. Since the weight function satisfies $r,r^{-1}\in L^\infty(\dR^n)$ the restrictions 
$r_\pm,r_\pm^{-1}$ belong to $L^\infty(\Omega_\pm)$ and hence the multiplication operators 
$R_\pm f_\pm=r_\pm f_\pm$ are isomorphisms in $L^2(\Omega_\pm)$ with inverses $R_\pm^{-1} f_\pm=r_\pm^{-1} f_\pm$,
$f_\pm\in L^2(\Omega_\pm)$.

Let us now assume that the coefficients $a_{jk}\in C^\infty(\dR^n)$ in \eqref{ellell} and their derivatives 
are uniformly continuous and bounded, and that (as before) $a\in L^\infty(\dR^n)$ is real valued.
An essential ingredient for the following considerations is that
by elliptic regularity and interpolation 
\begin{equation*}
\dom A=\dom T=H^2(\dR^n)\qquad\text{and}\qquad \cH_s=H^s(\dR^n),\qquad s\in [0,2],
\end{equation*}
holds; cf. \cite{A97,B65,LM72,W87}, \cite[Condition 3.1]{M10} and \eqref{a}, \eqref{t}. Here $H^s(\dR^n)$ is the Sobolev space or order $s$.
The spaces consisting of restrictions of functions from $H^s(\dR^n)$ onto 
$\Omega_\pm$ are denoted by $H^s(\Omega_\pm)$.
In the next lemma and remark we give simple sufficient conditions for the weight function $r$ such that
condition (I) in Theorem~\ref{bigthm} holds.

\begin{lemma}\label{sufflem}
Assume that for some $s\in(0,\frac{1}{2})$ the spaces $H^s(\Omega_+)$ and $H^s(\Omega_-)$ are invariant subspaces
of the multiplication operators $R_+$ and $R_-$, respectively. Then $H^s(\dR^n)$ is an invariant
subspace of the multiplication operator $R$, and condition {\rm (I)} in Theorem~\ref{bigthm}
is satisfied with $W$ replaced by $R$.
\end{lemma}

\begin{proof}
Let $s$ be as in the assumptions of the lemma and let $f\in H^s(\dR^n)$. Then
the restrictions $f_\pm$ of $f$ onto $\Omega_\pm$ are functions in $H^s(\Omega_\pm)$ 
and therefore, by assumption, the functions $g_\pm:=r_\pm f_\pm$ also belong to $H^s(\Omega_\pm)$. 
As $0<s<\frac{1}{2}$, the continuations $\widetilde g_\pm$ of $g_\pm$ by zero 
onto $\dR^n$ both are in $H^s(\dR^n)$; cf. \cite[Theorem~1.4.4.4 and Corollary~1.4.4.5]{G85}
and note that the proofs of these statements in \cite{G85} 
also cover the case of an unbounded domain with a compact smooth boundary.
Therefore $Rf=rf=\widetilde g_+ +\widetilde g_-\in H^s(\dR^n)$ and hence $H^s(\dR^n)$ is 
invariant for $R$. Furthermore, $R$ is an isomorphism in $L^2(\dR^n)$ 
and the estimate $R^2\geq \text{essinf}\, r^2 >0$ holds, i.e.,
$R$ possesses all the properties of the operator $W$ in condition (I) in Theorem~\ref{bigthm}.
\end{proof}

\begin{remark}\label{suffrem}
If, e.g., the function $r$ is equal to a (negative) constant outside some bounded
subset of $\dR^n$ and belongs to the H\"{o}lder space $C^{0,\alpha}(\dR^n)$ for some $\alpha>0$,
then it follows from
\cite[Theorem 1.4.1.1]{G85} and a similar argument as in the proof of Lemma~\ref{sufflem}
that $H^s(\Omega_\pm)$, $s\in(0,\alpha)$, are invariant subspaces of $R_\pm$.
\end{remark}

For completeness we state the following immediate consequence of Theorem~\ref{bigthm}
and Lemma~\ref{sufflem}.

\begin{corollary}\label{suffcor}
Assume that $\min\sigma_\ess(A)\leq 0$ and $R_\pm (H^s(\Omega_\pm))\subset H^s(\Omega_\pm)$ holds for some $s\in(0,\frac{1}{2})$.
Then the nonreal spectrum of $T$ is bounded.
\end{corollary}

In the following theorem we obtain more precise statements on the 
qualitative spectral properties of $T$. The proof is based on a multidimensional
variant of Glazmanns decomposition method from the theory of ordinary differential operators,
see, e.g., \cite{AG81,BT07,CL89,EE,N68}.

\begin{theorem}\label{rnthm}
Assume that $\min\sigma_\ess(A)\leq 0$. If $\rho(T)\not=\emptyset$, then the essential spectrum of $T$ is real, bounded from above, and 
$\sigma_\ess(T)\cap [0,\infty)\not=\emptyset$ holds. Moreover,
the nonreal spectrum of $T$ consists of normal eigenvalues which are symmetric
with respect to the real line and which may accumulate to points in $\sigma_\ess(T)$. 

If, in particular, $R_\pm (H^s(\Omega_\pm))\subset H^s(\Omega_\pm)$ holds for some $s\in(0,\frac{1}{2})$, then the assumption $\rho(T)\not=\emptyset$
is satisfied, the above assertions hold and the nonreal spectrum of $T$ is bounded. 
\end{theorem}

\begin{proof}
Besides the operators $A$ and $T$ we will make use of the selfadjoint elliptic differential operators
\begin{equation}\label{apm}
A_\pm f_\pm:=\ell(f_\pm),\qquad \dom A_\pm=H^2(\Omega_\pm)\cap H^1_0(\Omega_\pm),
\end{equation}
in $L^2(\Omega_\pm)$ and the weighted differential operators 
\begin{equation}\label{bpm}
B_\pm f_\pm:=\cL_\pm (f_\pm)=\frac{1}{r_\pm}\ell(f_\pm),\qquad \dom B_\pm=H^2(\Omega_\pm)\cap H^1_0(\Omega_\pm),
\end{equation}
which are selfadjoint in the weighted $L^2$-space $L^2(\Omega_\pm,\pm r_\pm)$, where the 
(positive definite)
scalar products $\langle\cdot,\cdot\rangle_\pm$ are defined by
\begin{equation}\label{asdfg}
\langle f_\pm,g_\pm\rangle_\pm:=
\int_{\Omega_\pm} f_\pm(x)\overline{g_\pm(x)}\,(\pm r_\pm(x))\,dx,\qquad f_\pm,g_\pm\in L^2(\Omega_\pm).
\end{equation}
Observe that the orthogonal sums $A_+\oplus A_-$ and $B_+\oplus B_-$ are selfadjoint
operators in $L^2(\Omega)$ and $L^2(\Omega,r)$, respectively. Furthermore, 
since the boundary $\cC$ is compact and smooth it can be shown that the 
resolvent differences
\begin{equation}\label{aaa}
(A-\lambda)^{-1}-\bigl((A_+\oplus A_-)-\lambda\bigr)^{-1},\qquad\lambda\in\rho(A)\cap\rho(A_+\oplus A_-),
\end{equation}
and
\begin{equation}\label{ttt}
(T-\lambda)^{-1}-\bigl((B_+\oplus B_-)-\lambda\bigr)^{-1},\qquad\lambda\in\rho(T)\cap\rho(B_+\oplus B_-),
\end{equation}
are compact operators in $L^2(\dR^n)$; cf. \cite{B62} and Theorem~\ref{resformel} below.

Recall that the spectra of $A_\pm$ are bounded from below and, moreover, 
as $\Omega_+$ is assumed to be bounded, the spectrum of $A_+$ consists of normal eigenvalues;
cf. Theorems~\ref{speca} and \ref{speca2}.
Furthermore, the differential operators in \eqref{apm} and \eqref{bpm} are connected via
\begin{equation}\label{bapm}
B_\pm=R_\pm^{-1}A_\pm,
\end{equation}
where $R_\pm$ are the multiplication operators with the functions $r_\pm$.
One verifies that for $\lambda\in
\rho(A_+)\cap\rho(B_+)$ the resolvents of $A_+$ and $B_+$
are connected via
\begin{equation*}
 (B_+-\lambda)^{-1}=(A_+-\lambda)^{-1}R_+-\lambda(A_+-\lambda)^{-1}(I-R_+)(B_+-\lambda)^{-1}
\end{equation*}
and since $(A_+-\lambda)^{-1}$ is compact the same holds for the resolvent of $B_+$;
cf. the proof of Theorem~\ref{boundedthm}. Thus the spectrum of
$B_+$ is also bounded from below and consists of normal eigenvalues which accumulate to $+\infty$. 

Next the spectrum of $B_-$ will be described in terms of the spectrum of $A_-$. 
Since the resolvent difference in \eqref{aaa} is compact and $\sigma_\ess(A_+)=\emptyset$
we conclude $$\sigma_\ess(A)=\sigma_\ess(A_+\oplus A_-)=\sigma_\ess (A_-).$$ 
Furthermore, the assumption $\min\sigma_\ess(A)\leq 0$ implies that
$A_-$ is bounded
from below with negative lower bound 
\begin{equation}\label{neglow}
\nu:=\min\sigma(A_-)\leq\min\sigma_\ess(A_-)\leq 0.
\end{equation}
Denote the usual scalar product in $L^2(\Omega_-)$ by $(\cdot,\cdot)_-$ and
let $\gamma$ be the supremum of the weight function $r_-$ on $\Omega_-$. Then 
we have $\gamma<0$ and $(-r_-(x))^{-1}\leq (-\gamma)^{-1}$ for all $x\in\Omega_-$.
Moreover, from the estimate
\begin{equation*}
(f_-,f_-)_-=\int_{\Omega_-}\frac{1}{-r_-(x)} \vert f_-(x)\vert^2\,(-r_-(x))\,dx\leq \frac{1}{-\gamma}\langle f_-,f_-\rangle_-,\quad f\in L^2(\Omega_-),
\end{equation*}
we obtain together with \eqref{asdfg} and \eqref{bapm} that
\begin{equation*}
\langle B_-f_-,f_-\rangle_-=(-A_-f_-,f_-)_-\leq -\nu(f_-,f_-)_-
\leq \frac{\nu}{\gamma}\langle f_-,f_-\rangle_-
\end{equation*}
holds for all $f_-\in\dom B_-$, i.e., the spectrum $\sigma(B_-)$ and the essential spectrum $\sigma_\ess(B_-)$ are 
bounded from above by the positive constant $\tfrac{\nu}{\gamma}$. 
Observe that $\max\sigma_\ess(B_-)\geq 0$ holds, since otherwise $\min\sigma_\ess(-B_-)$ is positive
and $A_-=(-R_-)(-B_-)$ implies that also $\min\sigma_\ess(A_-)$ is positive which contradicts \eqref{neglow}.

Summing up we have shown that the essential spectrum of $B_+\oplus B_-$ is real, bounded
from above and $\sigma_\ess(B_+\oplus B_-)\cap [0,\infty)\not=\emptyset$. 
Since the resolvent difference \eqref{ttt} is compact we obtain $\sigma_\ess(T)=\sigma_\ess(B_+\oplus B_-)=\sigma_\ess(B_-)$ which together with 
Corollary~\ref{suffcor} yields the statements.
\end{proof}


In the following we will show that the nonreal eigenvalues of $T$ and the corresponding eigenspaces can be
characterized with the help of Dirichlet-to-Neumann maps associated to the restrictions of the
elliptic differential expression $\cL$ on $\Omega_\pm$. For this, recall first that the mapping 
$C^{\infty}(\overline\Omega_\pm)\ni f_\pm\mapsto\{f_\pm\vert_{\cC},\tfrac{\partial f_\pm}{\partial\nu_\pm}\vert_\cC\}$
extends to a continuous surjective mapping
\begin{equation}\label{trace}
H^2(\Omega_\pm)\ni f_\pm \mapsto\left\{f_\pm\vert_{\cC},\frac{\partial f_\pm}{\partial\nu_\pm}\Bigl|_\cC\right\}
\in H^{3/2}(\cC)\times H^{1/2}(\cC),
\end{equation}
where $\frac{\partial f_\pm}{\partial\nu_\pm}\bigl|_\cC:=\sum_{j,k=1}^n a_{jk}\, \mathfrak n_{\pm,j}
\frac{\partial f_\pm}{\partial x_k}\bigl|_\cC$ and 
$\mathfrak n_\pm(x) = (\mathfrak n_{\pm,1}(x),\dots, \mathfrak n_{\pm,n}(x))$
is the unit vector at the point $x\in\cC$ pointing out of $\Omega_\pm$. The next 
simple lemma is based on a standard decomposition argument. For the convenience of the reader we provide a complete proof.

\begin{lemma}\label{dnlem}
For $\lambda\in\dC\backslash\dR$ and $\varphi\in H^{3/2}(\cC)$ there exist unique 
functions $f_{\pm,\lambda}(\varphi)\in H^2(\Omega_\pm)$ such that 
\begin{equation*}
\cL_\pm f_{\pm,\lambda}(\varphi)=\lambda f_{\pm,\lambda}(\varphi)\quad\text{and}\quad f_{\pm,\lambda}(\varphi)\vert_\cC=\varphi.
\end{equation*}
\end{lemma}

\begin{proof}
It is sufficient to show that for $\lambda\in\dC\backslash\dR$ the linear subspace  
\begin{equation}\label{cscs}
\cS:=\bigl\{h_+\oplus h_-\in H^2(\Omega_+)\oplus H^2(\Omega_-):h_+\vert_\cC=h_-\vert_\cC\bigr\}
\end{equation}
admits the direct sum decomposition
\begin{equation}\label{decos}
\cS=\bigl\{g_+\oplus g_-\in \cS:g_\pm\vert_\cC=0\bigr\}\,\dot +\,\bigl\{h_{+,\lambda}\oplus h_{-,\lambda}\in\cS: 
\cL_\pm h_{\pm,\lambda}=\lambda h_{\pm,\lambda} \bigr\}.
\end{equation}
In fact, it follows from \eqref{trace} that the trace map $h\mapsto h\vert_\cC$ defined on $\cS$ in \eqref{cscs} maps onto $H^{3/2}(\cC)$ and since
the first term on the right hand side of \eqref{decos} is its kernel it follows
that the trace map maps the second term on the right hand side of \eqref{decos} bijectively onto $H^{3/2}(\cC)$.

In order to prove the decomposition \eqref{decos} note first that the inclusion $\supset$ in \eqref{decos} holds. Hence it remains
to verify the inclusion $\subset$. For this let $h_+\oplus h_-\in\cS$
and $\lambda\in\dC\backslash\dR$ be fixed. 
Since the boundary $\cC$ of $\Omega_\pm$ is assumed to be compact and smooth it follows 
that the differential operators $B_\pm$ in \eqref{bpm} are defined on
\begin{equation}\label{domdom}
\dom B_\pm=H^2(\Omega_\pm)\cap H^1_0(\Omega_\pm)=\{f\in H^2(\Omega_\pm):f_\pm\vert_\cC=0\}.
\end{equation}
Hence the first set on the right hand side of \eqref{decos} coincides with $\dom(B_+\oplus B_-)$.  Since the spectrum of $B_+\oplus B_-$ is contained
in $\dR$ (see the proof of Theorem~\ref{rnthm}) $B_+\oplus B_- -\lambda$, $\lambda\in\dC\backslash\dR$, is a bijection from its domain onto $L^2(\dR^n)$. 
Thus there exists $g_+\oplus g_-\in\dom(B_+\oplus B_-)$
such that
\begin{equation*}
 (\cL_+-\lambda)h_+\,\oplus (\cL_--\lambda)h_- = (B_+-\lambda)g_+\,\oplus (B_--\lambda)g_-.
\end{equation*}
Therefore $\cL_\pm (h_\pm - g_\pm)= \lambda (h_\pm - g_\pm)$ and hence 
$$h_+\oplus h_- = g_+\oplus g_-\, +\, \bigl((h_+-g_+)\oplus (h_--g_-)\bigr)$$
shows that the inclusion $\subset$ in \eqref{decos} is also valid. The sum in \eqref{decos}
is direct since $\sigma(B_+\oplus B_-)\subset\dR$; indeed, each element in the intersection of
the sets on the right hand side of \eqref{decos} would be an eigenfunction of $B_+\oplus B_-$ 
corresponding to $\lambda\in\dC\backslash\dR$.
\end{proof}

For $\lambda\in\dC\backslash\dR$, $\varphi\in H^{3/2}(\cC)$ and 
$f_{\pm,\lambda}(\varphi)\in H^2(\Omega_\pm)$ as in Lemma~\ref{dnlem}
we define 
\begin{equation}\label{dnmap}
M(\lambda):H^{3/2}(\cC)\rightarrow H^{1/2}(\cC),\qquad \varphi\mapsto \frac{\partial f_{+,\lambda}(\varphi)}{\partial\nu_+}\Bigl|_\cC
+\frac{\partial f_{-,\lambda}(\varphi)}{\partial\nu_-}\Bigl|_\cC.
\end{equation}
Roughly speaking $M$ is the sum of the Dirichlet-to-Neumann maps associated to $\cL_\pm$ which map
the Dirichlet boundary values of solutions of $\cL_\pm f_\pm=\lambda f_\pm$ onto their Neumann boundary values.
A similar function in a ``definite'' setting appears also in \cite{R09}.
In the next theorem we show how the nonreal eigenvalues of $T$ can be described with the help of the function $M$.

\begin{theorem}\label{sigmat}
Let the operator function $\lambda\mapsto M(\lambda)$ be defined as in \eqref{dnmap} and assume that $\rho(T)\not=\emptyset$. Then
\begin{equation*}
\sigma(T)\cap(\dC\backslash\dR)=\bigl\{ \lambda\in\dC\backslash\dR:\ker M(\lambda)\not =\{0\}\bigr\}
\end{equation*}
and $\ker(T-\lambda)=\{f\in H^2(\dR^n): M(\lambda) f\vert_\cC=0\}$ for all $\lambda\in\dC\backslash\dR$.
\end{theorem}

\begin{proof}
Assume first that $\lambda\in\dC\backslash\dR$ belongs to the spectrum of $T$. Then, by Theorems~\ref{known2} and~\ref{rnthm}
the point $\lambda$ is a normal eigenvalue of $T$ and hence $\cL f=\lambda f$ holds for some nontrivial 
$f\in\dom T=H^2(\dR^n)$. In particular,
the restrictions $f_\pm$ of $f$ onto $\Omega_\pm$ belong to $H^2(\Omega_\pm)$ and we have
\begin{equation}\label{bcs}
\cL_\pm f_\pm=\lambda f_\pm,\quad
f_+\vert_\cC=f_-\vert_\cC,\quad\text{and}\quad \frac{\partial f_+}{\partial\nu_+}\Bigl|_\cC=-\frac{\partial f_-}{\partial\nu_-}\Bigl|_\cC.
\end{equation}
By \eqref{trace} we have $\varphi:=f_\pm\vert_\cC\in H^{3/2}(\cC)$ and hence 
$f_\pm=f_{\pm,\lambda}(\varphi)$ in the notation of
Lemma~\ref{dnlem}. The third property in \eqref{bcs} implies
\begin{equation*}
M(\lambda)\varphi=\frac{\partial f_{+,\lambda}(\varphi)}{\partial\nu_+}
+\frac{\partial f_{-,\lambda}(\varphi)}{\partial\nu_-}=\frac{\partial f_+}{\partial\nu_+}
+\frac{\partial f_-}{\partial\nu_-}=0
\end{equation*}
and hence $\varphi\in\ker M(\lambda)$.
Furthermore, $\varphi$ is nonzero, as otherwise $f_\pm\in H^2(\Omega_\pm)$ would be nontrivial 
solutions of the Dirichlet problems $\cL_\pm f_\pm=\lambda f_\pm$, $f_\pm\vert_\cC=0$, which do not exist due to 
$\lambda\not\in\dR$. In other words, since the selfadjoint operators $B_\pm$ in \eqref{bpm} 
do not have nonreal eigenvalues we conclude $\varphi\not=0$.

For the converse let $\lambda\in\dC\backslash\dR$ and $\varphi\in\ker M(\lambda)$ with $\varphi\not=0$. By Lemma~\ref{dnlem} 
there exist unique functions $f_{\pm,\lambda}(\varphi)\in H^2(\Omega_\pm)$ such that 
$\cL_\pm f_{\pm,\lambda}(\varphi)=\lambda f_{\pm,\lambda}(\varphi)$ and $f_{\pm,\lambda}(\varphi)\vert_\cC=\varphi$ hold. Since
$M(\lambda)\varphi=0$ we have
\begin{equation}\label{tzu}
\frac{\partial f_{+,\lambda}(\varphi)}{\partial\nu_+}=-\frac{\partial f_{-,\lambda}(\varphi)}{\partial\nu_-}.
\end{equation}
Define the function $f=f_+\oplus f_-\in L^2(\dR^n)$ by $f_\pm:=f_{\pm,\lambda}(\varphi)$ and let
$g\in\dom T$. Then $f\not=0$ and 
\begin{equation}\label{tfg}
\begin{split}
&\qquad\qquad[\cL f,g]-[f,Tg]=(\ell f,g)-(f,\ell g)\\
&=(\ell_+ f_+,g_+)_+ - (f_+,\ell_+ g_+)_+ + (\ell_- f_-,g_-)_- - (f_-,\ell_- g_-)_-,\\
\end{split}
\end{equation}
where $[\cdot,\cdot]$ is the indefinite inner product in \eqref{indefprod}, $(\cdot,\cdot)$ is the usual
scalar product in $L^2(\dR^n)$ and $(\cdot,\cdot)_\pm$ denote the scalar products in $L^2(\Omega_\pm)$.
Since the function $g\in\dom T$ satisfies 
\begin{equation*}
 g_+\vert_\cC=g_-\vert_\cC\qquad\text{and}\qquad\frac{\partial g_+}{\partial\nu_+}\Bigl|_\cC=-\frac{\partial g_-}{\partial\nu_-}\Bigl|_\cC
\end{equation*}
it follows from Green's identity, $f_+\vert_\cC=f_-\vert_\cC$, and 
\eqref{tzu} that \eqref{tfg} is equal to
\begin{equation*}
\biggl(f_+\vert_\cC,\frac{\partial g_+}{\partial\nu_+}\Bigl|_\cC \biggr)
-\biggl(\frac{\partial f_+}{\partial\nu_+}\Bigl|_\cC,g_+\vert_\cC\biggr)
+\biggl(f_-\vert_\cC,\frac{\partial g_-}{\partial\nu_-}\Bigl|_\cC \biggr)
-\biggl(\frac{\partial f_-}{\partial\nu_-}\Bigl|_\cC,g_-\vert_\cC\biggr)=0.
\end{equation*}
This is true for any $g\in\dom T$ and since $T$ is selfadjoint with respect to $[\cdot,\cdot]$ we conclude from $[\cL f,g]=[f,Tg]$ that $f\in\dom T$ and $Tf=\cL f$.
Moreover, from $\cL_\pm f_\pm=\lambda f_\pm$ we obtain $f\in\ker(T-\lambda)$, i.e., $\lambda$ is an eigenvalue of $T$
with corresponding eigenfunction $f$. 
\end{proof}

The next theorem provides a variant of Krein's formula which
shows how the resolvent of the indefinite elliptic operator
$T$ differs from the resolvent of the orthogonal sum of the weighted differential operators
\begin{equation*}
B_\pm f_\pm=\cL_\pm(f_\pm)=\frac{1}{r_\pm}\ell_\pm(f_\pm),\qquad \dom B_\pm=H^2(\Omega_\pm)\cap H^1_0(\Omega_\pm);
\end{equation*}
cf. \eqref{bpm} and \eqref{domdom}. The operators $T$ and $B_+\oplus B_-$ are viewed as operators in $L^2(\dR^n)$. 
We note first that the statements in 
Lemma~\ref{dnlem} and Theorem~\ref{sigmat} remain true if the set $\dC\backslash\dR$ is replaced by the resolvent set of the operator $B_+\oplus B_-$. 
This set contains $\dC\backslash\dR$ and may also contain subsets of the real line.
For $\lambda\in\rho(B_+\oplus B_-)$ define the mapping $\gamma(\lambda):L^2(\cC)\rightarrow L^2(\dR^n)$ by
\begin{equation*}
\gamma(\lambda)\varphi:= 
f_{+,\lambda}(\varphi)\oplus f_{-,\lambda}(\varphi),\qquad \dom\gamma(\lambda)=H^{3/2}(\cC),
\end{equation*}
where $f_{\pm,\lambda}(\varphi)$ are the unique solutions of $\cL_\pm u_\pm=\lambda u_\pm$, 
$u_\pm\vert_\cC=\varphi$; cf. Lemma~\ref{dnlem}. 

Theorem~\ref{resformel} is an indefinite variant of \cite[Theorem 4.4~(ii)]{AB09} and can
be proved in almost the same way. Therefore we only indicate some ideas of the proof and refer the
reader to \cite[$\S 4$]{AB09} for the details, see also \cite{BL07}. Recall that the multiplication operator 
$R$ is an isomorphism in $L^2(\dR^n)$.

\begin{theorem}\label{resformel}
For all $\lambda\in\rho(T)\cap\rho(B_+\oplus B_-)$ the difference of the resolvents of 
$T$ and $B_+\oplus B_-$ is a compact operator in $L^2(\dR^n)$ given by
\begin{equation}\label{resfor}
(T-\lambda)^{-1}-\bigl((B_+\oplus B_-)-\lambda\bigr)^{-1}=\gamma(\lambda)M(\lambda)^{-1}\gamma(\bar\lambda)^*R.
\end{equation}
\end{theorem}

\begin{proof}
A slight modification of \cite[Proposition 4.3]{AB09} yields that
$\gamma(\lambda)$ is a densely defined bounded operator from $L^2(\cC)$ into $L^2(\dR^n)$ and
that the adjoint operator $\gamma(\bar\lambda)^*:L^2(\dR^n)\rightarrow
L^2(\cC)$ has the property
\begin{equation*}
\gamma(\bar\lambda)^*R\bigl((B_+\oplus B_-)-\lambda\bigr)f=
-\frac{\partial f_+}{\partial\nu_+}\Bigl|_\cC
-\frac{\partial f_-}{\partial\nu_-}\Bigl|_\cC,
\end{equation*}
where $f=f_+\oplus f_-\in\dom (B_+\oplus B_-)$.
In particular, $\ran\gamma(\bar\lambda)^*\subset H^{1/2}(\cC)$ and together with Lemma~\ref{dnlem},
Theorem~\ref{sigmat} and \eqref{dnmap} we conclude that the right hand side in \eqref{resfor} is well defined
for all $\lambda\in\rho(T)\cap\rho(B_+\oplus B_-)$. The same reasoning as in the proof of 
\cite[Theorem 4.4~(ii)]{AB09} shows the relation \eqref{resfor} for the difference of the resolvents
of $T$ and $B_+\oplus B_-$ in $L^2(\dR^n)$. Moreover, it follows from \cite{LM72} in the same way as in 
\cite[Corollaries 3.6 and 4.6]{AB09}
that 
for $\lambda\in\rho(T)\cap\rho(B_+\oplus B_-)$
the closure of $M(\lambda)^{-1}$ in $L^2(\cC)$ is a compact operator in $L^2(\cC)$.
Therefore,
the right hand side of \eqref{resfor} is a compact operator in $L^2(\dR^n)$.
\end{proof}

\begin{remark}
We note that the resolvent difference in \eqref{resfor} is not only compact but
belongs to certain Schatten-von Neumann ideals that depend on the dimension $n$; 
cf. \cite{BLL10,B62,G84,M10}.
\end{remark}


\end{document}